\documentclass[10pt]{amsart}
\usepackage{amssymb}
\usepackage[varbb, varg, smallerops]{newtxmath}
\usepackage[scaled=0.87]{inconsolata} 
\usepackage[T1]{fontenc}
\usepackage{amsfonts}
\usepackage{comment}
\usepackage[a4paper,margin=1.1in]{geometry} 
\usepackage{tikz}
\usepackage{tikz-cd}
\usetikzlibrary{arrows}
\tikzcdset{arrow style=tikz,
diagrams={>={Stealth[round,length=4pt,width=6pt,inset=3pt]}}}
\usepackage{graphicx}
\usepackage[all]{xy}
\makeatletter
\newcommand*{\tightdisplaymath}{\abovedisplayskip\z@\belowdisplayskip\z@}

\author{Xiaofa Chen} 
\author{Mattia Ornaghi}


\thanks{X.C. is supported by “the Fundamental Research Funds for the Central Universities” (grant number: WK0010000093).\ 
M.O. was supported by the research project FARE 2018 HighCaSt (grant number: R18YA3ESPJ) and by ERC Advanced
Grant - CUP: G43C23001750006 - HE-ERC23ANEEM-01}

\newcommand{\vers}{A remark on fibrancy of ($\mbox{A$_{\infty}$Cat}$,$W^{\tiny\mbox{A}_{\infty}}_{\tiny\mbox{qe}}$)}
\title[\vers]{A remark on fibrancy of ($\mbox{A$_{\infty}$Cat}$,$W^{\tiny\mbox{A}_{\infty}}_{\tiny\mbox{qe}}$)}

\usepackage{hyperref}
\hypersetup{colorlinks=false}

\address{\parbox{0.9\textwidth}{University of Science and Technology of China\\
School of Mathematics Sciences\\
Hefei, 230026, Anhui Province, P.R.China}}
\email{xchen@imj-prg.fr}

\address{\parbox{0.9\textwidth}{Universit\`a degli Studi di Milano\\
Dipartimento di Matematica\\
Via Cesare Saldini 50, 20133 Milano, Italy}}
\email{mattia12.ornaghi@gmail.com}

\theoremstyle{definition}
\newtheorem{defn}{Definition}[section]

\newtheorem{thm}{Theorem}[section]
\newtheorem{pro}{Properties}[section]
\newtheorem{lem}[thm]{Lemma}

\newtheorem{namedthmA}{Theorem}
\newtheorem{namedcorA}{Corollary}

\theoremstyle{remark}
\newtheorem{rem}{Remark}[section]
\newtheorem{exmp}{Example}[section]

\newcommand{\Ain}{\mbox{A$_{\infty}$}}

\newcommand{\aCat}{\mbox{A$_{\infty}$Cat}}

\newcommand{\Id}{\mbox{Id}}

\newcommand{\A}{\mathscr{A}}

\newcommand{\B}{\mathscr{B}}

\newcommand{\Fun}{\mbox{Fun}}

\newcommand{\G}{\mathscr{G}}

\newcommand{\F}{\mathscr{F}}
\newcommand{\C}{\mathscr{C}}

\begin{document}

\date{\today}

\maketitle

\begin{abstract}
In \cite[\S6]{Pas} J. Pascaleff asked if ($\mbox{A$_{\infty}$Cat}$,$W^{\tiny\mbox{A}_{\infty}}_{\tiny\mbox{qe}}$) is a fibrant object in RelCat.\
In this note we prove the existence, in the category of (strictly unital) $\Ain$categories, of the pullback of a (strictly unital) $\Ain$functor, satisfying a particular property (denoted by F1), along any $\Ain$functor.\ As a consequence we provide a positive answer to Pascaleff question.
\end{abstract}

\section*{Introduction}

It is well known that the category of DG Categories, linear over a commutative ring, has a model structure \cite{Tab}.\ 
The weak equivalences are the quasi-equivalences and the fibrations are the DG functors $F:C\to D$ such that: 
\begin{itemize}
\item[(f1)] For every $x,y\in C$, the map of chain complexes
\begin{align*}
F: C(x,y)\to D(Fx,Fy)
\end{align*} 
is a degreewise surjection. 
\item[(F2)] The functor of linear categories $H^0(F)$ is an isofibration.\
\end{itemize}
This model structure provides a very solid description of the homotopy theory of DGCat.\\ 
On the other hand, the category $\Ain$Cat, of $\Ain$categories (i.e. the DG categories associative up to cohomology) with the $\Ain$functors, has no model structure.\
To be precise, the category of $\Ain$categories is not even complete, since it has no equalizers (see \cite{COS}).\\
However the category of $\Ain$categories can be equipped with a structure of relative category, whose weak equivalences are the quasi-equivalences.\ 
Briefly, an $\Ain$functor $\F:\A\to\B$, is a \emph{quasi-equivalence} if it induces an equivalence of graded categories $[\F]:H(\A)\to H(\B)$.\ 
We denote by $W^{\tiny\mbox{A}_{\infty}}_{\tiny\mbox{qe}}$ the set of quasi-equivalences.\\ 
Recently, using the structure of relative category, James Pascaleff proved that the homotopy theories of differential graded categories and of $\Ain$categories, over a field, are equivalent at the $(\infty,1)$-categorical level (see \cite{Pas}).\ 
In his paper it is remained open the question if the relative category ($\mbox{A$_{\infty}$Cat}$,$W^{\tiny\mbox{A}_{\infty}}_{\tiny\mbox{qe}}$) is fibrant.\ If this was true, than he pointed out that his proof could be simplified.\ In this note we prove: 

\begin{namedthmA}\label{Pasca}
The relative category ($\mbox{A$_{\infty}$Cat}$,$W^{\tiny\mbox{A}_{\infty}}_{\tiny\mbox{qe}}$) of $\Ain$categories is a fibrant object in RelCat.
\end{namedthmA}

First we need to choose a subclass $F^{\tiny\mbox{A}_{\infty}}$ of $\Ain$functors playing the role of \emph{fibrations}.\ There is a very natural choice of fibrations considering the work of Lef\`evre-Hasegawa on $\Ain$algebras (see \cite{LH}): 
a \emph{fibration} is an $\Ain$functor satisfying the conditions (F1)\footnote{Clearly every epimorphism is splitting if $R$ is a field.} and (F2) below.\\ 

We say that an A$_{\infty}$functor $\mathscr{F}:\A\to \B$ has the properties:
\begin{itemize}
\item[(F1)]  If, for any pair of objects $x,y\in\A$, 
\begin{align*}
\mathscr{F}^1:\A(x,y)\to \B(\F_0(x),\F_0(y))
\end{align*}
is a graded-split surjection.
\item[(F2)] If  
\begin{align*}
[\mathscr{F}]:H^0(\A)\to H^0(\B)
\end{align*}
is an isofibration.
\end{itemize}
From now on, we denote by $\F:\A \twoheadrightarrow \B$ an $\Ain$functor $\F$ satisfying (F1).\ 
We will use this notation throughout.\\
Before continuing, let us clarify that in Theorem \ref{Pasca}, with "$\Ain$Cat", we mean the category of strictly unital flat $\Ain$categories with strictly unital $\Ain$functors (cf.~Remark \ref{DefClassica} and Definition \ref{Afununit}).\ Clearly conditions (F1) and (F2) makes sense for non unital $\Ain$functors between non unital non flat $\Ain$categories, linear over a commutative ring.\\

As Pascaleff pointed out, using \cite[Definition 3.1]{Mei}, Theorem \ref{Pasca} holds if the pullback of a fibration of A$_{\infty}$categories along any A$_{\infty}$functor exists and is a fibration, and also that the pullback of an acyclic fibration (i.e. a fibration which is also a quasi-equivalence) is an acyclic fibration.\ So it follows directly from the followings three (more general) results: Theorem \ref{teoremone}, Theorem \ref{teoremino} and Corollary \ref{corollario}, which we will prove in this paper:

\begin{namedthmA}\label{teoremone}
Let $\A$, $\A'$, $\A''$ be three A$_{\infty}$categories, $\F:\A\twoheadrightarrow\A'$ an A$_{\infty}$functor satisfying (F1), and $\G:\A''\to\A'$ an $\Ain$functor.\ 
The diagram:
\begin{equation}\label{princip}
\xymatrix{
\A\ar@{->>}[r]^{\F}&\A'&\ar[l]_{\G}\A''
}
\end{equation}
has a pullback in the category of (strictly unital) $\Ain$categories.\\ 
Denoting these data by $(\mathscr{P},\alpha,\beta)$, fitting the following diagram
\begin{equation}\label{eq:1}
\xymatrix{
\mathscr{P}\ar@{-->}[r]^{\alpha}\ar@{-->}[d]^{\beta}&\A''\ar[d]^-{\G}\\
\A\ar@{->>}[r]^-{\F}&\A'
}
\end{equation}
we have that $\alpha$ satisfies (F1).  
\end{namedthmA}

\begin{namedcorA}\label{corollario}
Let $\A$, $\A'$, $\A''$, $\F:\A\twoheadrightarrow\A'$ and $\G:\A''\to\A'$ be as in Theorem \ref{teoremone}, if they are strictly unital then (\ref{princip}) has a pullback in the category of strictly unital $\Ain$categories.
\end{namedcorA}

\begin{namedthmA}\label{teoremino}
Let $\A$, $\A'$, $\A''$, $\F:\A\twoheadrightarrow\A'$ and $\G:\A''\to\A'$ be as in Theorem \ref{teoremone}.\
\begin{itemize}
\item[i.] 
If $\F$ satisfies (F2) then $\alpha$ satisfies (F2).\
\item[ii.] If $\F$ satisfies (F2) and it is a quasi-equivalence then $\alpha$ is a quasi-equivalence.
\end{itemize}
\end{namedthmA}

\subsection*{State of art and related works}
Thanks to the works of Lef\`evre-Hasegawa \cite{LH} and Vallette \cite{Val}, it is known that the category of $\Ain$algebras, linear over a field, is a fibration category in the sense of \cite{Mei}.\ In this case, the weak equivalences are the $\Ain$quasi-isomorphisms and the fibrations are degreewise surjections (i.e. condition (F1)).\
In a few words, our Theorem \ref{teoremone} is a generalization to the (non flat) categories of Lef\`evre's \cite[Th\'eorème 1.3.3.1 (A)]{LH} for algebras.\ 
The main point here is the passage from (graded) $R$-modules to (graded) $R$-quivers.\ Let us note that the hypothesis of splitting in condition (F1) is essential (see section \ref{subquiv}) but this is not required in the definition of fibration for DG categories.\ 
For this reason, one expects the right notion of "fibration" in ($\mbox{A$_{\infty}$Cat}$,$W^{\tiny\mbox{A}_{\infty}}_{\tiny\mbox{qe}}$) to be the functors satisfying conditions (f1) and (F2) (not (F1)).\ So, our proof of the fibrancy of 
($\mbox{A$_{\infty}$Cat}$,$W^{\tiny\mbox{A}_{\infty}}_{\tiny\mbox{qe}}$) 
cannot be extended to categories linear over a commutative ring in this sense.\ 

On the other hand, using Pascaleff's procedure and \cite{COS2} we can prove that the homotopy theories of differential graded categories and of $\Ain$categories, over a commutative ring, 
are equivalent at the $(\infty,1)$-categorical level (see \cite[Theorem A]{COS2}).\\
Note that one can describe the homotopy theory of $\Ain$categories, over a commutative ring, 
via semi-free resolutions avoiding the fibrant point of view (cf. \cite{Orn2}).

We conclude by saying that the category $\aCat$ has a weak symmetric monoidal structure (see \cite{Orn4}) which deserves to be investigated at the $\infty$-level (see \cite[A.4.4]{Pas2}).\ We hope that this paper could give a contribution (considering $\Ain$categories linear over a field) in this direction.

\subsection*{Acknowledgements}
The authors thank Bernhard Keller for the useful conversations about Lefèvre-Hasegawa thesis and James Pascaleff for sharing his interest in the writing of this notes.


\section{Graded quivers and A$_{\infty}$structures}
We give a short background on graded quivers and $\Ain$structure.\
We fix a commutative ring $R$.\ All categories are assumed to be small and linear over $R$.
\begin{defn}[Graded Quiver]
A \emph{graded quiver} $\mathtt{Q}$ consists of a set of objects $\mbox{Ob}(\mathtt{Q})$ and a graded $R$-module $\mathtt{Q}(x_1,x_2)$, for every pair of objects $x_1,x_2\in\mbox{Ob}(\mathtt{Q})$. 
\end{defn}

\begin{exmp}
A graded $R$-module can be considered as a graded quiver with one object.  
\end{exmp}

\begin{defn}[Formal morphism of Graded Quivers]\label{formorph}
Let $\mathtt{Q}_1$ and $\mathtt{Q}_2$ be two graded quivers.\
A \emph{formal morphism $\mathsf{F}$ between $\mathtt{Q}_1$ and $\mathtt{Q}_2$} is the datum: 
$\mathsf{F}:=\mathcal{f}\mathsf{F}^n\mathcal{g}_{n\ge0}$ such that:
\begin{itemize}
\item[1.] If $n=0$, $\mathsf{F}^0:\mbox{Ob}(\mathtt{Q}_1)\to \mbox{Ob}(\mathtt{Q}_2)$ is a map of sets.
\item[2.] For every $n\ge1$ and $x_0,...,x_n\in \mbox{Ob}(\mathtt{Q})^{\times n+1}$, 
\begin{align*}
\mathsf{F}^n:\mathtt{Q}_1(x_{n-1},x_n)\otimes...\otimes \mathtt{Q}_1(x_0,x_1)\to \mathtt{Q}_2(\mathsf{F}_0(x_0),\mathsf{F}_0(x_n))[1-n]
\end{align*}
is a map of graded $R$-modules.
\end{itemize}
\end{defn}

\begin{rem}
Definition \ref{formorph} is a generalization of the definition of {formal diffeomorphism} by Seidel \cite[(1c)]{Sei}.\ A \emph{formal diffeomorphism} is a formal morphism such that $\mathtt{Q}_1=\mathtt{Q}_2$ and $\mathsf{F}_1$ 
is linear automorphism of $\mathtt{Q}_1(x,y)$, for every $x,y\in\mbox{Ob}(\mathtt{Q}_1)$. 
\end{rem}

\begin{defn}[Composition of formal morphisms]
Let $\mathsf{F}:\mathtt{Q}_1\to \mathtt{Q}_2$ and $\mathsf{G}:\mathtt{Q}_2\to \mathtt{Q}_3$ be two formal morphisms.\ The \emph{composition} $\mathsf{G}\cdot\mathsf{F}:\mathtt{Q}_1\to\mathtt{Q}_3$ is the formal morphism defined as follows
\begin{align*}
(\mathsf{G}\cdot\mathsf{F})^{n}(q_n,...,q_1):=\displaystyle\sum_{i_r+...+i_1=n} \mathsf{G}_r\big(\mathsf{F}_{i_r}(q_n,...,q_{n-i_r}),...,\mathsf{F}_{i_1}(q_{i_1},...,q_1)\big)
\end{align*}
for every $n\ge0$ and $q_j\in \mathtt{Q}_1(x_{j},x_{j+1})$.
\end{defn}

\begin{rem}
We fix a graded quiver $\mathtt{Q}$, a useful example of formal morphism is 
the \emph{identity functor} of $\mathtt{Q}$, defined as follows:
\begin{align*}
\mbox{Id}^n_{\mathtt{Q}}(q_n,...,q_1):=
\begin{cases}
q_1\mbox{, if $n=1$}\\
0\mbox{, otherwise}.
\end{cases}
\end{align*}
More generally, a formal morphism $\mathsf{F}$ such that $\mathsf{F}^{n\ge 2}=0$ is said to be a \emph{functor of graded quivers}.\\
The class of (graded) quivers with formal morphisms form a category which we denote by ${\textbf{Qu}}_{\tiny\mbox{Form}}$.\ It is important to note that this is different from the category of (graded) quivers with functors (of graded quivers), which is commonly denoted by ${\textbf{Qu}}$ (see e.g. \cite[\S1.2]{COS}).
\end{rem}

We fix two graded quivers $\mathtt{Q}_1$, $\mathtt{Q}_2$ and two formal morphisms $\mathsf{F},\mathsf{G}:\mathtt{Q}_1\to \mathtt{Q}_2$, as follows: 
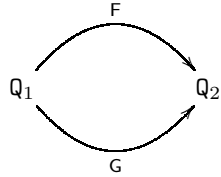
\begin{align*}
\xymatrix@R=0.2em{
&  &\\
\mathtt{Q}_1\ar@/^2pc/[rr]^-{\mathsf{F}}\ar@/_2pc/[rr]_-{\mathsf{G}}&&\mathtt{Q}_2\\
&&
}
\end{align*}
Considering the formal morphisms as 1-morphisms, we can define the 2-morphisms between 1-morphisms as follows.
\begin{defn}[Prenatural transformation of Graded Quivers]
A \emph{prenatural transformation} $T$, of degree $g$, from $\mathsf{F}$ to $\mathsf{G}$ is given by the datum 
$T=\mathcal{f}T^n\mathcal{g}_{n\ge 0}$, such that:
\begin{itemize}
\item[1.] If $n=0$, for every $x\in\mbox{Ob}(\mathtt{Q}_1)$, we have:
$$T^0(x):\mathsf{F}^0(x)\to\mathsf{G}^0(x)\in\mathtt{Q}_2\big(\mathsf{F}^0(x),\mathsf{G}^0(x)\big).$$
\item[2.] For every $n\ge1$ and $x_0,...,x_n\in \mbox{Ob}(\mathtt{Q}_1)^{\times n+1}$:
\begin{align*}
T^n:\mathtt{Q}_1(x_{n-1},x_n)\otimes...\otimes\mathtt{Q}_1(x_0,x_1)\to \mathtt{Q}_2\big(\mathsf{F}^0(x_0),\mathsf{G}^0(x_n) \big)[g-n]
\end{align*}
is a map of graded $R$-modules.
\end{itemize}
\end{defn}
A prenatural transformation from $\mathsf{F}$ to $\mathsf{F}$ is called a \emph{prenatural $\mathsf{F}$-endotransformation}.\\

We consider the following situation: let $\mathsf{F}:\mathtt{Q}_1\to\mathtt{Q}_2$ and $\mathsf{G}:\mathtt{Q}_2\to\mathtt{Q}_3$ be two formal morphisms.\\
Given a prenatural $\mathsf{F}$-endotransformation $D'$ and a prenatural $\mathsf{G}$-endotransformation $D$:
\begin{align}\label{compo}
\xymatrix@R=1em{
&\ar@{=>}[dd]^{\tiny D'}& &\ar@{=>}[dd]^{D}&\\
\mathtt{Q}_1\ar@/^2pc/[rr]^-{\mathsf{F}}\ar@/_2pc/[rr]_-{\mathsf{F}}&&\mathtt{Q}_2\ar@/^2pc/[rr]^-{\mathsf{G}}\ar@/_2pc/[rr]_-{\mathsf{G}}&&\mathtt{Q}_3\\
&&&&
}
\end{align}
we can define the composition of $D$ and $D'$ as follows.

\begin{defn}[Composition of prenatural transformations]\label{compderiv}
Given $D$ and $D'$ two prenatural transformations as in (\ref{compo}).\ The \emph{composition} $D\circ D'$ is a prenatural $(\mathsf{G}\cdot\mathsf{F})$-endotransformation, defined as:
\begin{align}\label{zorgo}
(D\circ D')^n(q_n,...,q_1):=\displaystyle\sum_{i_r+...+i_1=n} (-1)^{\star} {D}^r\big(\mathsf{F}^{i_r}(q_n,...,q_{n-i_r}),...,D'^{i_k}(...),..,\mathsf{F}^{i_1}(q_{i_1},...,q_1)\big)
\end{align}
for every $q_n,...,q_1\in\mathtt{Q}_1(x_j,x_{j+1})$.
\end{defn}

Taking a graded quiver $\mathtt{Q}$ the Bar construction $B_{\infty}$ gives rise to a (reduced) graded cocomplete cocategory $\big(B_{\infty}(\mathtt{Q}),\Delta\big)$.\ Given two graded quivers $\mathtt{Q}_1,\mathtt{Q}_2$, we have:
\begin{align}\label{f2}
\xymatrix{
{\mathcal{f}\mbox{Formal morphisms: $\mathtt{Q}_1\to\mathtt{Q}_1$} \mathcal{g}}\ar@_{<->}[r]^-{1:1}&{\mathcal{f}\mbox{Functors of cocategories: $(B_{\infty}(\mathtt{Q}_1),\Delta) \to (B_{\infty}(\mathtt{Q}_2),\Delta)$}\mathcal{g} }
}
\end{align}
Fixing two graded quivers and two formal morphisms as follows
\[
\xymatrix{
\mathtt{Q}_1\ar@/^/[r]^{F}\ar@/_/[r]_{F'}&\mathtt{Q}_2
}
\]
we have a corresponding diagram in the category of (reduced) graded cocomplete cocategories: 
\[
\xymatrix{
\big(B_{\infty}(\mathtt{Q}_1),\Delta\big)\ar@/^/[r]^{B_{\infty}(F)}\ar@/_/[r]_{B_{\infty}(F')}&\big(B_{\infty}(\mathtt{Q}_2),\Delta\big).
}
\]
We have a bijection:
\begin{align}\label{f3}
\xymatrix{
{\mathcal{f}\mbox{Prenatural transformations: $F\to F'$}\mathcal{g}}\ar@_{<->}[r]^-{1:1}&\mathcal{f}\mbox{$\big(B_{\infty}(F),B_{\infty}(F')\big)$-coderivations}\mathcal{g}.
}
\end{align}
Given a prenatural transformation $D$ the $\big(B_{\infty}(F),B_{\infty}(F')\big)$-coderivation\footnote{By \cite[Lemma 2.2]{Orn1} we have that $B_{\infty}(G\cdot F)=B_{\infty}(G)\cdot B_{\infty}(F)$} $B_{\infty}(D)$ is defined in \cite[\S2.3]{Orn1}.\ Taking another prenatural transformation $D'$ such that $B_{\infty}(D')$ is a $\big(B_{\infty}(G),B_{\infty}(G')\big)$-coderivation as follows
\[
\xymatrix{
\big(B_{\infty}(\mathtt{Q}_1),\Delta\big)\ar@/^/[r]^{B_{\infty}(F)}\ar@/_/[r]_{B_{\infty}(F')}&\big(B_{\infty}(\mathtt{Q}_2),\Delta\big)\ar@/^/[r]^{B_{\infty}(G)}\ar@/_/[r]_{B_{\infty}(G')}&\big(B_{\infty}(\mathtt{Q}_3).
}
\]
It is easy to define the $(B_{\infty}(G\cdot F),B_{\infty}(G'\cdot F'))$-coderivation $B_{\infty}(D)\circ B_{\infty}(D')$.\ 
The signs $\star$ of (\ref{zorgo}) are defined via $B_{\infty}(D)\circ B_{\infty}(D')$.\\
In particular, for fixed a graded quiver $\mathtt{Q}$, we have a bijection:
\begin{align}\label{f4}
\xymatrix{
{\mathcal{f}\mbox{Prenatural $\Id_{\mathtt{Q}}$-endotransformations}\mathcal{g}}\ar@_{<->}[r]^-{1:1}&\mathcal{f}\mbox{$\big(\Id_{B_{\infty}(\mathtt{Q})},\Id_{B_{\infty}(\mathtt{Q})}\big)$-coderivations}\mathcal{g}.
}
\end{align}
We recall that a \emph{DG structure} on $B_{\infty}(\mathtt{Q})$ is a $(\Id_{B_{\infty}},\Id_{B_{\infty}})$-coderivation $d$ such that $d\circ d=0$.

\begin{defn}[$\Ain$structure]\label{ACat}
Let $\mathtt{Q}$  be a graded quiver.\ 
An $\Ain$structure on $\mathtt{Q}$ is a prenatural $\mbox{Id}_{\mathtt{Q}}$-endotransformation $D$ such that 
$B_{\infty}(D)\circ B_{\infty}(D)=0$.
\end{defn}

\begin{exmp}
A graded $R$-module equipped with an $\Ain$structure is said to be an \emph{$\Ain$algebra}.  
\end{exmp}

The category of non unital $\Ain$categories is equivalent to the category of (reduced) cocomplete DG-cocategories.\
Namely we have the followings bijection.\
Let $\mathtt{Q}$ be a graded quiver:
\begin{align}\label{f1}
\xymatrix{
{\mbox{$\mathcal{f}$$\Ain$structure on $\mathtt{Q}$$\mathcal{g}$}}\ar@_{<->}[r]^-{1:1}&{\mbox{$\mathcal{f}$DG structure on $B_{\infty}(\mathtt{Q})$$\mathcal{g}$}}
}
\end{align}
The bijections $(\ref{f1})$, $(\ref{f2})$ and $(\ref{f3})$ are well known to the experts, an explicit description can be found in 
\cite[\S 2.1,\S 2.2,\S 2.3]{Orn1}.
\begin{rem}\label{DefClassica}
We point out that Definition \ref{ACat} is not the usual definition of $\Ain$category.\  
To get to "classical" definition of $\Ain$category one has to assume that:
\begin{itemize} 
\item[A1.] $D^{0}(x)=0$, for every object $x$.\
\end{itemize}
Some authors call \emph{flat $\Ain$categories} the $\Ain$categories with $D^0=0$.
\end{rem}

\begin{rem}\label{exxpli}
Let $(\A,m_{\A})$ be an $\Ain$category as in Definition \ref{ACat} such that $m^0=0$.\ 
Unwinding formula ($\ref{zorgo}$) and taking into account the signs (see \cite[\S2.1]{Orn1}, \cite[\S1.2]{COS2}) we get the formula: 

\begin{align*}
\displaystyle\sum^n_{m=1}\sum^{n-m}_{d=0}(-1)^{\dagger_d} m_{\A}^{n-m+1}(f_n,...,f_{d+m+1},m_{\A}^{m}(f_{d+m},...,f_{d+1}),f_d,...,f_1)=&0,
\end{align*}
where $\dagger_d=\mbox{deg}(f_d)+...+\mbox{deg}(f_1)-d$.\
See \cite[Definition 1.1]{COS} \cite[(1a)]{Sei}, \cite[Définition 1.2.1.1]{LH}).
\end{rem}

\begin{defn}[Strictly unital $\Ain$category]
A \emph{strictly unital} $\Ain$category $\A$ is an $\Ain$category such that, for every object $a\in\mbox{Ob}(\A)$, there exists a morphism $1_a\in\A(a,a)$ satisfying the following:
\begin{itemize}
\item[u1.] $D^{n\not=2}(f_n,...,1_a,...,f_1)=0$, $\forall f_1,...,f_n\in\A$.
\item[u2.] $D^2(1_{a_2},f)=D^2(f,1_{a_1})=f$, for every $f:a_1\to a_2$.
\end{itemize}
\end{defn}

\subsection{Composition functors}

Let $\mathtt{A},\mathtt{B}$ two graded quivers.\ We can form the graded quiver
\begin{align*}
\Fun(\mathtt{A},\mathtt{B}).
\end{align*}
The objects are the formal morphisms and the morphisms are the prenatural transformations.\\
Given a formal morphism $\mathsf{F}:\mathsf{A}\to\mathsf{B}$ we have
\begin{align*}
\mathscr{L}_{\mathsf{F}}:\Fun(\mathtt{A},\mathtt{A})(\mbox{Id}_{\mathtt{A}},\mbox{Id}_{\mathtt{A}})\to \Fun(\mathtt{A},\mathtt{B})(\mathsf{F},\mathsf{F})
\end{align*} 
Given $D\in \Fun(\mathtt{A},\mathtt{A})(\mbox{Id}_{\mathtt{A}},\mbox{Id}_{\mathtt{A}})$ we have:
\begin{align*}
\mathscr{L}_{\mathsf{F}}(D)^n(a_n,...,a_1):=\displaystyle\sum_{}(-1)^{\heartsuit} \mathsf{F}^j\big(a_n,...,D_r(...),...,a_1\big).
\end{align*}
Here $\heartsuit$ can be found in \cite[(1d)]{Sei}.\\
On the other hand, 
\begin{align*}
\mathscr{R}_{\mathsf{F}}:\Fun(\mathtt{B},\mathtt{B})(\mbox{Id}_{\mathtt{B}},\mbox{Id}_{\mathtt{B}})\to \Fun(\mathtt{A},\mathtt{B})(\mathsf{F},\mathsf{F})
\end{align*}
Given $D'\in \Fun(\mathtt{B},\mathtt{B})(\mbox{Id}_{\mathtt{B}},\mbox{Id}_{\mathtt{B}})$ we have:
\begin{align*}
\mathscr{R}_{\mathsf{F}}(D')^n(b_n,...,b_1):=\displaystyle\sum_{}D'^r \big(\mathsf{F}^{i_r}(b_{n},...,b_{n-i_r}),...,\mathsf{F}^{i_1}(a_{i_1},...,a_1)\big).
\end{align*}

We list some useful properties from \cite[pp.12]{Sei}:
\begin{pro}
Let $\mathtt{A}$, $\mathtt{B}$ be two graded quivers and a formal morphism:
\begin{align*}
\xymatrix{
\mathtt{A}\ar[r]^{\mathsf{F}}&\mathtt{B}.
}
\end{align*}
Given $D_{\mathtt{A}}\in \Fun(\mathtt{A},\mathtt{A})(\mbox{Id}_{\mathtt{A}},\mbox{Id}_{\mathtt{A}})$ and $D_{\mathtt{B}}\in \Fun(\mathtt{B},\mathtt{B})(\mbox{Id}_{\mathtt{B}},\mbox{Id}_{\mathtt{B}})$.\ It is not difficult to prove:
\begin{itemize} 
\item[1.] $\mathscr{L}_{\mathsf{F}}(D_{{\mathtt{B}}}\circ D_{\mathtt{A}})=\mathscr{L}_{\mathsf{F}}(D_{{\mathtt{B}}}) \circ D_{{\mathtt{A}}}$.
\item[2.] $\mathscr{R}_{\mathsf{F}}(D_{{\mathtt{B}}}\circ D_{{\mathtt{B}}})=\mathscr{R}_{\mathsf{F}}(D_{{\mathtt{B}}}) \circ D_{{\mathtt{A}}}$.
\end{itemize}
Given another quiver $\mathtt{C}$, a formal morphism $\mathsf{G}$ as follows:
\begin{align*}
\xymatrix{
\mathtt{A}\ar[r]^{\mathsf{F}}&\mathtt{B}\ar[r]^{\mathsf{G}}&\mathtt{C}.
}
\end{align*}
and $D_{\mathtt{C}}\in\Fun(\mathtt{C},\mathtt{C})(\Id_{\mathtt{C}},\Id_{\mathtt{C}})$, one can prove:
\begin{itemize}
\item[3.] There is an equality:
$$\mathscr{R}_{\mathsf{F}}\big( \mathscr{R}_{\mathsf{G}}(D_{\mathtt{C}}) \big)=\mathscr{R}_{\mathsf{G}\cdot \mathsf{F}}(D_{\mathtt{C}}),$$
in $\Fun(\mathtt{A},\mathtt{C})(\mathsf{G}\cdot \mathsf{F}, \mathsf{G}\cdot \mathsf{F})$.
\item[4.] There is an equality:
$$\mathscr{L}_{\mathsf{G}}\big( \mathscr{L}_{\mathsf{F}}(D_{\mathtt{A}}) \big)=\mathscr{L}_{\mathsf{G}\cdot \mathsf{F}}(D_{\mathtt{A}}),$$
in $\Fun(\mathtt{A},\mathtt{C})(\mathsf{G}\cdot \mathsf{F}, \mathsf{G}\cdot \mathsf{F})$.
\item[5.] There is an equality:
$$\mathscr{L}_{\mathsf{G}}\big( \mathscr{R}_{\mathsf{F}}(D_{\mathtt{B}})\big)=\mathscr{R}_{\mathsf{F}}\big( \mathscr{L}_{\mathsf{G}}(D_{\mathtt{B}}) \big),$$
In $\Fun(\mathtt{A},\mathtt{C})(\mathsf{G}\cdot \mathsf{F}, \mathsf{G}\cdot \mathsf{F})$.
\end{itemize}
\end{pro}

\begin{defn}[$\Ain$functor]\label{afun}
Given two $\Ain$categories $(\mathtt{A},D_{\mathtt{A}})$ and $(\mathtt{B},D_{\mathtt{B}})$ a formal morphism $\mathsf{F}$ is an \emph{$\Ain$functor} if $\mathscr{L}_{\mathsf{F}}(D_{\mathtt{A}})=\mathscr{R}_{\mathsf{F}}(D_{\mathtt{B}})$.
\end{defn}

\begin{rem}
As in Remark \ref{exxpli} we can make explicit Definition \ref{afun}.\ 
Let $\F$ be a formal morphism between two $\Ain$categories $(\A,m_{\A})$ and $(\B,m_{\B})$.\ 
\begin{align*}
\displaystyle\sum^n_{r=1}\sum_{s_1,...,s_r}&m_{\B}^{r}\big(\F^{s_r}(f_n,...,f_{n-s_r+1}),...,\F^{d}(f_d,...,f_1)\big)=\\
&=\displaystyle\sum_{m,n}(-1)^{\dagger_d} \F^{n-m+1}\big(f_n,...,f_{d+m+1},m_{\A}^{m}(f_{d+m},...,f_{d+1}),f_d,...,f_1 \big)
\end{align*}
See \cite[(1b)]{Sei}, \cite[Definition 1.2.1]{COS} and \cite[Définition 1.2.1.2]{LH}.
\end{rem}

\begin{defn}\label{Afununit}
A \emph{strictly unital} $\Ain$functor $\F:\A\to\B$ is an $\Ain$functor between strictly unital $\Ain$categories such that:
\begin{itemize}
\item[fu1.] $\F^1(1_x)=1_{\F^0(x)}$, $\forall x\in\mbox{Ob}(\A)$.
\item[fu2.] $\F^n(...,1_x,...)=0$.
\end{itemize}
\end{defn}

\section{Sub quivers generated by objects}\label{subquiv}

In this section we define the notion of \emph{sub quiver generated by objects}, which can be considered as a particular case of \cite[\S 3]{KM}, and we define the (categorical) product of two formal morphisms.\ The goal of this subsection is Theorem \ref{teorema2}.\ Namely, we found an $\Ain$category making diagram (\ref{eq:4}) commutative.\ This category is a candidate to be the pullback of diagram (\ref{eq:1}).\\

Through this section we assume the following setup:
\begin{align*}
\F:\A\twoheadrightarrow\A' 
\end{align*}
is an $\Ain$functor satisfying (F1).\ It means that, fixing two objects $x,y\in\mbox{Ob}(\A)$, we have
a splitting short exact sequence:
\begin{equation}\label{equi3}
\xymatrix{
0\ar[r]&\mbox{Ker}\big(\F^1(x,y)\big)\ar@{_(->}[r]_-{i^1}&\ar@/_/[l]_-{r^1}\A(x,y)\ar@{->>}[r]_-{\F^1}&\ar@/_/[l]_-{s^1}\A'\big(\F^0(x),\F^0(y)\big)\ar[r]&0
}
\end{equation}
in the category of graded $R$-modules.\\
Every morphism of $\A(x,y)$ can be written as:
\begin{align*}
\A(x,y)&\simeq \mbox{Ker}\big(\F^1(x,y)\big)\oplus \A'\big(\F^0(x),\F^0(y)\big).\\
f &\mapsto \big(r^1(f),\F^1(f)\big)\\
i^1(g)+s^1(h)&\leftarrow (g,h).
\end{align*}
We define the graded quiver $(K\times \A)_{S}$ as follows:
\begin{itemize}
\item[1.] $\mbox{Ob}\big( (K\times \A)_{S} \big):=\mathcal{f} (a,\F^0(a))\in \mbox{Ob}(\A)\times\mbox{Ob}(\A')\mathcal{g}\leftrightarrow \mbox{Ob}(\A)$,
\item[2.] Given two objects $(a_1,\F^0(a_1),(a_2,\F^0(a_2)\in \mbox{Ob}\big( (K\times \A)_{S} \big)$, we define 
\begin{align*}
(K\times\A')_{S}\big( (a_1,\F^0(a_1)),(a_2,\F^0(a_2)) \big):=\mbox{Ker}\big(\F^1(a_1,a_2)\big)\oplus \A'\big(\F^0(a_1),\F^0(a_2)\big).
\end{align*}
\end{itemize}
It is easy to see that $\A$ is isomorphic to $(K\times \A')_{S}$ as graded quivers.\\ 
On the other hand, thanks to bijection (\ref{f1}), $B_{\infty}(\A)$ is isomorphic to $B_{\infty}\big((K\times \A')_{S}\big)$ as (reduced) cocomplete cocategory.\\

We have an endomorphism $\gamma$ of the (reduced) graded cocomplete cocategory $B_{\infty}(\A)$ defined as follows:

\begin{itemize}
\item[1.] $\gamma^0:=\Id_{\tiny\mbox{Ob}(\A)}$,
\item[2.] Given a morphism $f_n[1]\otimes...\otimes f_1[1]\in B_{\infty}(\A)$ we define:
\begin{align*}
\gamma^1(f_n[1]\otimes...\otimes f_1[1]):=\displaystyle\sum^{n}_{k=1}\sum_{\star^n_k}s^1_{i_1}\F^{i_1}(f_n,...,f_{i_n-i_1+1})[1]\otimes...\otimes s^1_{i_k}\F^{i_k}(f_{i_k},...,f_1)[1]\big).
\end{align*}
Where $\displaystyle\sum_{\star^n_k}$ denotes  $\displaystyle\sum_{i_k+...+i_1=n}$ such that there exists at least one $i_j>1$, and 
$s^1_{i_j}$ denotes a splitting as in (\ref{equi3}).\ 
Namely, if $\F^{i_j}(...):x_{i_j}\to y_{i_j}$ then $s^1_{i_j}$ is a splitting of $$\F^1:\A(x_{i_j},y_{i_j})\twoheadrightarrow\A'(\F^0(x_{i_j}),\F^0(y_{i_j})).$$
\end{itemize}
The verification that $\gamma$ is a functor of cocategories is left to the reader.\ 
Note that, if $\F^0$ is a bijection, then we can make $s^1$ a functor of (graded) categories.\ In this case
\begin{align*}
\gamma=B_{\infty}(s^1)\big(B_{\infty}(\F)-B_{\infty}(\F^1)\big).
\end{align*}
Here $\F^1:\A\to\A'$ is the formal morphism of graded quivers given by $\mathcal{f}\F^0,\F^1,0,0,...\mathcal{g}$.\ One can find the explicit formulas of $B_{\infty}(s^1)$, $B_{\infty}(\F^1)$ and $B_{\infty}(\F)$ in \cite[Example 2.2.1]{Orn1} and \cite[\S2.2]{Orn1}\footnote{Note that $B_{\infty}$ sends a functor $\mathsf{F}$ of (graded) category to a functor $B_{\infty}(\mathsf{F})$ of (graded) cocategories.\ Moreover, if $\mathsf{F}$ is an A$_{\infty}$functor (between $\Ain$categories) then $B_{\infty}(\mathsf{F})$ is a DG functor (between DG cocategories)}.\

Despite $\A$ and $\A'$ are $\Ain$categories, in general $\F^1$ is not an $\Ain$functor.\\ 
We have an automorphism of $B_{\infty}(\A)$ defined as follows:
\begin{align*}
\Phi:= \mbox{Id}_{B_{\infty}(\A)}+\gamma.
\end{align*}
First we need to say that, fixed two objects $x,y\in \mbox{Ob}\big(B_{\infty}(\A)\big)$, we have a complete filtration:
\begin{align*}
0\subset F_1(x,y)\subset ... \subset F_n(x,y) \subset ... \subset B_{\infty}(\A)(x,y).
\end{align*}
Where $F_n(x,y)$ denotes the graded complex of the morphisms of the form $f_j[1]\otimes...\otimes f_1[1]$, where $j\le n$ and $f_i\in\A$ for every $1\le i \le j$.\\
The inverse of $\Phi$ is defined as follows:
\begin{align*}
\Psi(f_n[1]\otimes...\otimes f_1[1]):=\big(\mbox{Id}_{B_{\infty}(\A)}-\gamma+...+(-1)^n \gamma^{\cdot n}\big) (f_n[1]\otimes...\otimes f_1[1]).
\end{align*}
It is not difficult prove that the following diagrams:
\begin{align}\label{toror}
\xymatrix{
B_{\infty}(\A)\ar[d]_{\Phi}^{\cong}\ar[rr]^{B_{\infty}(\F)}&&B_{\infty}(\A')\ar@{=}[d]\\
B_{\infty}(\A)\ar[rr]_-{B_{\infty}(\F^1)}&&B_{\infty}(\A')
}
\mbox{ \hspace{3cm} }
\xymatrix{
B_{\infty}(\A)\ar[rr]^{B_{\infty}(\F)}&&B_{\infty}(\A')\ar@{=}[d]\\
B_{\infty}(\A)\ar[u]^{\Psi}_{\cong}\ar[rr]_-{B_{\infty}(\F^1)}&&B_{\infty}(\A')
}
\end{align}
are commutative in the category of (reduced) cocomplete cocategories.\\
Since the category of $\Ain$categories are in 1:1 correspondence with the category of reduced cocomplete DG cocategories, 
diagram $(\ref{princip})$ has pullback if and only if the following diagram 
\begin{equation}\label{pllk2}
\xymatrix{
&&B_{\infty}(\A'')\ar[d]^-{B_{\infty}(\G)}\\
B_{\infty}(\A) \ar@{->>}[rr]^-{\tiny B_{\infty}(\F)}&&B_{\infty}(\A')
}
\end{equation}
has pullback in the category of reduced cocomplete DG cocategories.\\
On the other hand, we can define a differential $\hat{D}:=\Psi\cdot D_{\A}\cdot\Phi$ on $B_{\infty}(\A)$, making the diagrams in ($\ref{toror}$) commutative in the category of reduced cocomplete DG cocategories.\\
It follows that $(\ref{princip})$ has pullback if the following diagram:
\begin{equation}\label{tryu}
\xymatrix{
&&\big(B_{\infty}(\A''),\Delta,D_{\A''}\big)\ar[d]^-{B_{\infty}(\G)}\\
\big(B_{\infty}(\A),\Delta,\hat{D}\big) \ar@{->>}[rr]^-{\tiny B_{\infty}(\F^1)}&&\big(B_{\infty}(\A'),\Delta,D_{\A'}\big)
}
\end{equation}
has pullback in the category of the DG reduced cocomplete cocategories.\\
Using the isomorphism of graded quivers 
$$\A\simeq \big( K\times \A'\big)_{S}$$
we can put an $\Ain$structure ${m}_{(K\times A')_{S}}$ on the graded quiver $\big( K\times \A'\big)_{S}$.\\
So $(\ref{eq:3})$ has limit in the category of $\Ain$categories iff the following diagram
\begin{equation}\label{eq:3}
\xymatrix{
&\A''\ar[d]^-{\G}\\
(K\times\A')_{S} \ar@{->>}[r]^-{\tiny\mbox{pr}^1_{\A'}}&\A'
}
\end{equation}
has pullback in $\aCat$.\ Here $\mbox{pr}^1_{\A'}$ is the strict $\Ain$functor (defined by $\F_1$) and $(K\times\A')_{S}$ has the $\Ain$structure ${m}_{(K\times A')_{S}}$.\ Since $\mbox{pr}^1_{\A'}$ is an $\Ain$functor wrt the $\Ain$structure ${m}_{(K\times A')_{S}}$ we have:
\begin{align*}
\mbox{pr}^1_{\A'}\big(m_{(K\times A')_{S}}^n\big((k_n,a'_n),...,(k_1,a'_1) \big)\big):=m_{\A'}^n(a_n,...,a_1).
\end{align*}

\begin{rem}
Despite $\A$, $\A'$ and $\A''$ are $\Ain$categories, in general, we cannot put an $\Ain$structure on $K$ making $\A\simeq (K\times \A')_S$ where $K\times\A'$ is the categorical product in $\aCat$.\ 
It happens if the sequence $(\ref{equi3})$ splits as DG complex not as graded complex.\
\end{rem}

\begin{rem}
If $\A$, $\A'$ and $\A''$ are $\Ain$algebras then we can take directly the graded $R$-module $K\times\A'$, we don't need the subquiver $(K\times\A')_{S}$.\ In this case $(\ref{tryu})$ corresponds to Lefèvre-Hasegawa \cite[Lemme 1.3.3.3]{LH}.
\end{rem}

\subsection{The product of two formal morphisms between two graded quivers}
We define the graded quiver $(K\times\A'')_{S'}$ as follows:
\begin{itemize}
\item[1.] $\mbox{Ob}\big( (K\times\A'')_{S'}\big):=\mathcal{f}\mbox{$(x,y)\in\mbox{Ob}(\A)\times\mbox{Ob}(\A'')$ such that $\F^0(x)=\G^0(y)$} \mathcal{g}$,
\item[2.] $(K\times\A'')_{S'}\big( (x_1,y_1),(x_2,y_2)\big):=K(x_1,x_2)\oplus\A''(y_1,y_2)$.
\end{itemize}
The quiver $(K\times\A'')_{S'}$ is a sub quiver (generated by a set of objects) of the quiver $K\times\A''$ which is the categorical product of the two quivers $K$ and $\A''$.\ We can define the formal morphism 
\begin{align*}
\mbox{Id}_{K}\times \G: K\times \A''\to K\times \A'
\end{align*} 
as follows:
\begin{itemize}
\item[P1.] Given $(x,y)\in \mbox{Ob}(\A)\times\mbox{Ob}(\A'')$ we have:
\begin{align*}
(\mbox{Id}_{K}\times \G)^0(x,y):=\big(x,\G^0(x)\big).
\end{align*}
\item[P2.] For every $n>0$:
\begin{equation}\label{morphi}
\big(\mbox{Id}_{K}\times \G\big)^{n}\big( (k_n,a''_n),...,(k_1,a_1'')\big):=
\begin{cases}
\big(k_1,\G^1(a''_1)\big), \mbox{if $n=1$}\\
\big(0,\G^n(a''_n,...,a''_1)\big), \mbox{if $n\ge2$}.
\end{cases}
\end{equation}
For every $(k_j,a''_j)\in K(x_j,x_{j+1})\oplus\A''(y_j,y_{j+1})$, $1\ge j\ge n$.
\end{itemize}

We need a Lemma first.

\begin{lem}\label{principone}
$\mbox{Im}\big(\mbox{Id}_{K}\times\G |_{\tiny(K\times\A'')_{S'}}\big)\subset (K\times\A')_{S}$.
\end{lem}

\begin{proof}
An object in $\big(K\times\A''\big)_{S'}$ is of the form $(x,y)$ such that $\F^0(x)=\G^0(y)$.\\
By definition $(\mbox{Id}_{K}\times \G)^0(x,y)=(x,\G^0(y))=(x,\F^0(x))\subset \mbox{Ob}\big((K\times\A')_{S}\big)$ and we are done.
\end{proof}

We have the following commutative diagram in the category of graded quivers (with formal morphisms):
\begin{equation}\label{eq:4}
\xymatrix{
(K\times\A'')_{S'}\ar[d]_{\tiny\mbox{Id}_{K}\times \G}\ar[r]^-{\tiny\mbox{pr}^1_{\A''}}&\A''\ar[d]^-{\G}\\
(K\times\A')_{S} \ar@{->>}[r]^-{\tiny\mbox{pr}^1_{\A'}}&\A'
}
\end{equation}
or, equivalently, we have the following commutative diagram:
\begin{equation*}
\xymatrix{
B_{\infty}\big((K\times\A'')_{S'}\big)\ar@{->>}[rr]^-{\tiny B_{\infty}(\mbox{pr}^1_{\A''})}\ar[d]^{B_{\infty}(\tiny\mbox{Id}_{K}\times \G)}&&B_{\infty}(\A'')\ar[d]^-{B_{\infty}(\G)}\\
B_{\infty}\big((K\times\A')_{S}\big) \ar@{->>}[rr]^-{\tiny B_{\infty}(\mbox{pr}^1_{\A'})}&&B_{\infty}(\A')
}
\end{equation*}
in the category of cocomplete graded cocategories.\ Note that $\mbox{pr}_{\A''}^1$ is strict.\\
We recall that $B_{\infty}(\A'')$, $B_{\infty}\big((K\times\A')_{S}\big)$ and $B_{\infty}(\A')$ have a DG structure (see (\ref{tryu})), or equivalently $\A'',\A'$ and $(K\times\A')_S$ have an $\Ain$structure.\\
In the next subsection we equip the graded quiver $(K\times\A'')_{S'}$ with an $\Ain$structure making (\ref{eq:4}) commutative in the category of $\Ain$categories.

\subsection{The $\Ain$structure on $(K\times\A'')_{S'}$}
The goal of this subsection is the following result:
\begin{thm}\label{teorema2}
The graded quiver $(K\times\A'')_{S'}$ has an $\Ain$structure making ($\ref{eq:4}$) commutative. 
\end{thm}

The proof of Theorem \ref{teorema2} is divided in two steps: first we prove the existence of a prenatural $\mbox{Id}_{(K\times\A'')_{S'}}$-endotransformation 
\begin{align*}
\tilde{D}^n: (K\times\A'')_{S'}\big((x_{n-1},y_{n-1}),(x_n,y_n)\big)\otimes (K\times\A'')_{S'}\big((x_1,y_1),(x_2,y_2) &\big)\to\\  
\to(K\times\A'')_{S'}\big( (x_1,y_1),(x_n,y_n)\big),
\end{align*}
such that 
\begin{align}\label{condizione1}
\mathscr{L}_{\tiny\mbox{Id}_{K}\times\G}(\tilde{D})^{n}\big((k_n,a''_n),...,(k_1,a''_1)\big)=\mathscr{R}_{\tiny\mbox{Id}_{K}\times\G}(D_{\A})^n\big((k_n,a''_n),...,(k_1,a''_1)\big)
\end{align}
and
\begin{align}\label{condizione2}
\mathscr{L}_{\tiny\mbox{pr}^1_{\A''}}(\tilde{D})^{n}\big((k_n,a''_n),...,(k_1,a''_1)\big)=\mathscr{R}_{\tiny\mbox{pr}^1_{\A''}}(D_{\A''})^n\big((k_n,a''_n),...,(k_1,a''_1)\big)
\end{align}
For every $n\ge 0$.\ Then we will prove that $\tilde{D}$ is an $\Ain$structure, i.e. $\tilde{D}\circ\tilde{D}=0$.\\
\\
We start with the following calculation.
\begin{exmp}\label{cefe}
We denote $\tilde{D}^1(k_1,a''_1):=(\tilde{k},\tilde{a}'')$.\\
Let us find $\tilde{k}\in K(x_1,x_2)$ and $\tilde{a}''\in \A''(y_1,y_2)$ such that:
\begin{align}\label{refgr}
\begin{cases}
\big(\mathscr{L}_{\tiny\Id_K\times\G}(\tilde{D}) - \mathscr{R}_{\tiny\Id_K\times\G}(m_{(K\times\A')_{S}})\big)^1(k_1,a_1'')=0\\
\big(\mathscr{L}_{\tiny\mbox{pr}^1_{\A''}}(\tilde{D}) - \mathscr{R}_{\tiny\mbox{pr}^1_{\A''}}(m_{\A''})\big)^1(k_1,a_1'')=0.
\end{cases}
\end{align}
By the second equation of (\ref{refgr}) we must have:
\begin{align*}
0&=\big(\mathscr{L}_{\tiny\mbox{pr}^1_{\A''}}(\tilde{D}) - \mathscr{R}_{\tiny\mbox{pr}^1_{\A''}}(m_{\A''})\big)^1(k_1,a_1'')\\
&=\mbox{pr}^1_{\A''}\big(\tilde{D}_1(k_1,a''_1)\big)- m_{\A''}^1\big( \mbox{pr}^1_{\A''}(k_1,a''_1)\big)\\
&=\mbox{pr}^1_{\A''}(\tilde{k}_1,\tilde{a}''\big)- m_{\A''}^1(a''_1).
\end{align*}
It means $\tilde{a}''=m_{\A''}^1(a''_1)$.\ On the other hand, by the first equation of ($\ref{refgr}$):
\begin{align*}
(0,0)&=\big(\mathscr{L}_{\tiny\tiny\Id_K\times\G}(\tilde{D}) - \mathscr{R}_{\tiny\tiny\Id_K\times\G}(m_{{(K\times\A')_{S}}})\big)^1(k_1,a_1'')\\
&=(\Id_K\times\G)^1\big(\tilde{D}^1(k_1,a''_1)\big)- m_{{(K\times\A')_{S}}}^1\big( k_1,\G^1(a''_1))\big)\\
&=(\tilde{k},\G^1(\tilde{a}'')\big)- \big(r^1\big(m_{\tiny{(K\times\A')_{S}}}^1(k_1,\G^1(a''_1))\big) , m_{\A'}^1 (\G^1(a''_1)) \big)\\
&=(\tilde{k},\G^1\big( m_{\A''}^1(a''_1)\big)\big)- \big(r^1\big(m_{\tiny{(K\times\A')_{S}}}^1(k_1,\G^1(a''_1))\big) , m_{\A'}^1 (\G^1(a''_1)) \big)\\
&=(\tilde{k} - r^1\big(m_{\tiny{(K\times\A')_{S}}}^1(k_1,\G^1(a''_1))\big),\G^1( m_{\A''}^1(a''_1))-m_{\A'}^1 (\G^1(a''_1))\big).
\end{align*}
Since $\G$ is an $\Ain$functor $\G^1( m_{\A''}^1(a''_1))-m_{\A'}^1 (\G^1(a''_1))=0$, so $\tilde{k}:=r^1\big(m_{\tiny{(K\times\A')_{S}}}^1(k_1,\G^1(a''_1))\big)$.\\
Now let us prove that $(\tilde{D}\circ\tilde{D})^1\big( (k_1,a''_1)\big)=0$.\ 
We calculate:
\begin{align*}
(\tilde{D}\circ\tilde{D})^1\big( (k_1,a''_1)\big)&=\tilde{D}^1\big(\tilde{D}^1\big((k_1,a''_1)\big)\big)\\
&=\tilde{D}^1\big( r^1\big(m_{\tiny{(K\times\A')_{S}}}^1(k_1,\G^1(a''_1))\big), m_{\A''}^1(a'') \big)\\
&=\big(r^1\big(m_{(K\times\A')_{S}}^1\big(r^1\big(m_{\tiny{(K\times\A')_{S}}}^1(k_1,\G^1(a''_1))\big),\G^1(m_{\A''}^1(a''_1)))\big), 0 \big)\\
&=\big(r^1 m_{(K\times\A')_{S}}^1(m_{(K\times\A')_{S}}^1(k_1,\G^1(a''_1))) ,0\big)\\
&=(0,0).
\end{align*}
\end{exmp}

\begin{proof}[Proof of Theorem \ref{teorema2}]
As we said before we divide this proof in two steps:
\begin{itemize}
\item[Step 1.] We want to define $\tilde{D}^n$ such that ($\ref{condizione1}$) and ($\ref{condizione2}$) hold.\ 
By Example \ref{cefe} we know how to define $\tilde{D}_1$.\ 
Suppose we can define $\tilde{D}^j$ for $1\ge j \ge n-1$.\ 
We need to define $\tilde{D}^n$ satisfying ($\ref{condizione1}$) and ($\ref{condizione2}$).\ 
As before, we use the notation: 
\begin{align*}
\tilde{D}^n\big((k_n,a''_n),...,(k_1,a''_1)\big):=(\tilde{k},\tilde{a}'').
\end{align*}
Fixed $\big((k_n,a''_n),...,(k_1,a''_1)\big)\in(K\times\A'')_{S'}$.\\
It easy to prove that we must have:
\begin{align*}
\tilde{a}'':= m_{\A''}^n(a''_n,...,a''_1)
\end{align*}
in order to verify ($\ref{condizione2}$).\ 
This is true for every positive integer $j\leq n$ and it means
\begin{align*}
\mbox{pr}^1_{\A''}\big( \tilde{D}_j((k_j,a''_j),...,(k_1,a''_1))\big):= m_{\A''}^j(a''_j,...,a''_1).
\end{align*}
Now we want to find $\tilde{k}$ in order to verify ($\ref{condizione1}$).\\
First we note:
\begin{align}\label{condizione3}
\mbox{pr}^1_{\A''}\big(\mathscr{L}_{\tiny\mbox{Id}_{K}\times\G}(\tilde{D})&-\mathscr{R}_{\tiny\mbox{Id}_{K}\times\G}(m_{(K\times \A')_{S}})\big)^n\big((k_n,a''_n),...,(k_1,a''_1)\big)\\
=&\big(\mathscr{L}_{\G}(m_{\A''})-\mathscr{R}_{\G}(m_{\A'})\big)^n(a''_n,...,a''_1)
=0.
\end{align}
Now we take the following equation:
\begin{align}\label{condizione12}
(\mathscr{L}_{\tiny\mbox{Id}_{K}\times\G}\tilde{D}-&\mathscr{R}_{\tiny\mbox{Id}_{K}\times\G}m_{(K\times \A')_{S}})^n\big((k_n,a''_n),...,(k_1,a''_1)\big)=\\
=&(\Id_K\times\G)^1(\tilde{D}^n\big((k_n,a''_n),...,(k_1,a''_1)\big)) + \psi^n\big((k_n,a''_n),...,(k_1,a''_1)\big)\\
=&(\Id_K\times\G)^1(\tilde{k},m_{\A''}^1(a''_n,...,a''_1)) + \psi^n\big((k_n,a''_n),...,(k_1,a''_1)\big) \\
=&\big(\tilde{k},\G^1(m_{\A''}^1(a''_n,...,a''_1))\big) + \psi^n\big((k_n,a''_n),...,(k_1,a''_1)\big).
\end{align}
Here we note that $\psi^n\big((k_n,a''_n),...,(k_1,a''_1)\big)$ is fully determined since it involves only the formal morphisms $\Id_K\times\G$, $\tilde{D}^{j\le n-1}$, $m_{(K\times \A')_{S'}}$.\\
We can define 
\begin{align*}
\tilde{k}:=\mbox{pr}^1_{K}\big(\psi^n\big((k_n,a''_n),...,(k_1,a''_1)\big)\big).
\end{align*}
By definition of $\tilde{k}$ and ($\ref{condizione3}$) we can define $\tilde{D}_{n}$ such that ($\ref{condizione1}$) and ($\ref{condizione2}$) hold.
\item[Step 2.] It remains to prove that $\tilde{D}$ defines an $\Ain$structure.\ In other words we need to show that 
\begin{align*}
(\tilde{D}\circ\tilde{D})^{n}\big((k_n,a''_n),...,(k_1,a''_1)\big)=0.
\end{align*}
We proved in Example \ref{cefe} that $(\tilde{D}\circ\tilde{D})^{1}=0$.\ 
Now we suppose that $(\tilde{D}\circ\tilde{D})^{j\le n-1}=0$.\\
As in the previous step, we use the following notation:
\begin{align*}
(\tilde{D}\circ\tilde{D})^{n}\big((k_{n},a''_{n}),...,(k_1,a''_1)\big)=(\tilde{k},\tilde{a}'').
\end{align*}
\end{itemize} 
We use the fact that $\mathscr{L}_{\F}(\tilde{D}\circ\tilde{D})=\mathscr{R}_{\F}({D'}\circ{D'})$, ($\ref{condizione1}$) and ($\ref{condizione2}$).\ We have:
\begin{align*}
0&=\mathscr{R}_{\tiny\mbox{pr}^1_{\A''}}({m_{\A''}}\circ{m_{\A''}})^{n}\big( (k_{n},a''_{n}),...,(k_1,a''_1)\big)\\
&=\mathscr{L}_{\tiny\mbox{pr}^1_{\A''}}(\tilde{D}\circ\tilde{D})^{n}\big( (k_{n},a''_{n}),...,(k_1,a''_1)\big)\\
&=(\mbox{pr}^1_{\A''})^1(\tilde{k},\tilde{a}'')\\
&=\tilde{a}''
\end{align*}
On the other hand we have:
\begin{align*}
(0,0)&=\mathscr{R}_{\tiny \Id_K\times\G}({m^{\A}}\circ{m^{\A}})^{n}\big( (k_{n},a''_{n}),...,(k_1,a''_1)\big)\\
&=\mathscr{L}_{\tiny\Id_K\times\G}(\tilde{D}\circ\tilde{D})^{n}\big( (k_{n},a''_{n}),...,(k_1,a''_1)\big)\\
&=(\Id_K\times\G)^1(\tilde{k},0) + \displaystyle\sum_{r\le n-1} (\Id_K\times\G)^j\big(...,(\tilde{D}\circ\tilde{D})^r,...\big)\\
&=(\tilde{k},0)
\end{align*}
the last equality follows by induction that $(\tilde D\circ \tilde D)^{r}=0$ because $r\leq n-1$ and we are done.
\end{proof}

\section{Proof of Theorems \ref{teoremone}, Corollary \ref{corollario} and Theorem \ref{teoremino}}
We claim that $(K\times\A'')_{S'}$ with the $\Ain$structure defined in Theorem \ref{teorema2} is the pullback of diagram (\ref{princip}).\ 
We need to show that, given an $\Ain$category $\C$ and two $\Ain$functors $\mathscr{L}:\C\to\A''$ and $\mathscr{I}:\C\to (K\times\A'')_{S'}$ such that, $\G\cdot\mathscr{L}=\mbox{pr}^1_{\A'}\cdot\mathscr{I}$, there exists a unique $\Ain$functor $\mathscr{N}$ such that the following diagram:
\[
\xymatrix{
\mathscr{C}\ar@{-->}[dr]^{\exists!\mathscr{N}}\ar@/_/[ddr]_{\mathscr{I}}\ar@/^/[drr]^{\mathscr{L}}&&\\
&(K\times\A'')_{S'}\ar[d]^{\tiny\mbox{Id}_{K}\times \G}\ar[r]_-{\tiny\mbox{pr}^1_{\A''}}&\A''\ar[d]^-{\G}\\
&(K\times\A')_{S} \ar@{->>}[r]_-{\tiny\mbox{pr}^1_{\A'}}&\A'
}
\]
commutes.\\
We suppose that $\mathscr{N}$ is the map induced forgetting the $\Ain$structure.\ 
Namely, $\mathscr{N}$ is the formal morphism defined as follows:
\begin{align}\label{pulbk}
\mathscr{N}_n(c_n,...,c_1):=\big( \mbox{pr}^1_{K}(\mathscr{I}_{n}(c_n,...,c_1)),\mathscr{L}_{n}(c_n,...,c_1)\big).
\end{align}
To prove Theorem \ref{teoremone} it remains to show that $\mathscr{N}$ is an $\Ain$functor i.e. $\mathscr{L}_{\mathscr{N}}(D_{\C})=\mathscr{R}_{\mathscr{N}}(\tilde{D})$.

\begin{lem}
The formal morphism $\mathscr{N}$ defined in (\ref{pulbk}) is an $\Ain$functor.
\end{lem}

\begin{proof}
It is easy to see that:
\begin{align*}
\mathscr{L}_{\mathscr{N}}(D_{\mathscr{C}})^1(c_1)=\mathscr{R}_{\mathscr{N}}(\tilde{D})^1(c_1)
\end{align*}
For every morphism $c_1\in\C$.\\
We use the following notation:
\begin{align*}
\phi^n(c_n,...,c_1):=\big(\mathscr{L}_{\mathscr{N}}(D_{\C})-\mathscr{R}_{\mathscr{N}}(\tilde{D}) \big)^n(c_n,...,c_1)=(\tilde{k},\tilde{a}'').
\end{align*}
We have $\phi\in\mbox{Fun}\big(\C,(K\times\A'')_{S'}\big)(\mathscr{N},\mathscr{N})$, so we are in the following situation:
\begin{align*}
\xymatrix@R=1em{
&\ar@{=>}[dd]^{\phi}& &&\\
\C\ar@/^2pc/[rr]^-{\mathscr{N}}\ar@/_2pc/[rr]_-{\mathscr{N}}&&(K\times\A'')_{S'}\ar[rr]_{\tiny\Id\times\G}&&(K\times\A')_{S}\\
&&&&
}
\end{align*}
We have $\mathscr{L}_{\tiny\Id\times\F}(\phi)\in \mbox{Fun}\big(\C,(K\times\A')_{S}\big)\big((\Id_K\times\G)\cdot\mathscr{N},(\Id_K\times\G)\cdot\mathscr{N}\big)=\mbox{Fun}\big(\C,(K\times\A')_{S}\big)\big(\mathscr{I},\mathscr{I}\big)$.\\
It is not hard to prove the following 
\begin{align*}
\mathscr{L}_{\tiny\Id\times\G}(\phi):=&\mathscr{L}_{\tiny\Id\times\G}\big(\mathscr{L}_{\mathscr{N}}(D_{\mathscr{C}})-\mathscr{R}_{\mathscr{N}}(\tilde{D})\big)=\mathscr{L}_{\tiny\Id\times\G}\big(\mathscr{L}_{\mathscr{N}}(D_{\mathscr{C}})\big)-\mathscr{L}_{\tiny\Id\times\G}\big(\mathscr{R}_{\mathscr{N}}(\tilde{D})\big)\\
=&\mathscr{L}_{(\tiny\Id\times\G)\cdot\mathscr{N}}(D_{\mathscr{C}})-\mathscr{R}_{\tiny(\Id\times\G)\cdot\mathscr{N}}({D}_{\A})-\big(\mathscr{R}_{\mathscr{N}}(\mathscr{L}_{\tiny\Id\times\G}(\tilde{D})-\mathscr{R}_{\tiny\Id\times1\G}(D_{\A})) \big)\\
=&\mathscr{L}_{\mathscr{I}}(D_{\mathscr{C}})-\mathscr{R}_{\mathscr{I}}({D}_{\A})\\
=&(0,0).
\end{align*}
Now we are ready to prove that $(\tilde{k},\tilde{a}'')=(0,0)$.\
Supposing that $\phi^{j\le n-1}=0$, we have:
\begin{align*}
0=\mbox{pr}^1_{K}\big(\big(\mathscr{L}_{\tiny\Id\times\G}(\phi)\big)^n(c_n,...,c_1)\big)=\mbox{pr}^1_{K}(\phi^n(c_n,...,c_1))=\tilde{k}.
\end{align*}
On the other hand we recall that 
$$\mbox{pr}^1_{\A''}(\tilde{D}^n((k_n,a''_n),...,(k_1,a''_1))=m_{\A''}^n(a''_n,...,a''_1)$$
and 
$$\mbox{pr}^1_{\A''}(\mathscr{N}^n(c_n,...,c_1))=\mathscr{L}^n(c_n,...,c_1).$$
It means:
\begin{align*}
\tilde{a}'':=\mbox{pr}^1_{\A''}\big(\phi_n(c_n,...,c_1)\big)=&\mbox{pr}^1_{\A''}\big(\mathscr{L}_{\mathscr{N}}(D_{\mathscr{C}})^n(c_n,...,c_1)-\mathscr{R}_{\mathscr{N}}(\tilde{D})^n(c_n,...,c_1)\big)\\
&=\mathscr{L}_{\mathscr{L}}(D_{\mathscr{C}})^n(c_n,...,c_1)-\mathscr{R}_{\mathscr{L}}({D}_{\A''})^n(c_n,...,c_1)=0.
\end{align*}
So $\phi^n=0$ and we prove that $\mathscr{N}$ is an $\Ain$functor. 
\end{proof}

\begin{proof}[Proof of Corollary \ref{corollario}]
We have that $\A$, $\A'$, $\A''$, $\F$ and $\G$ are strictly unital, 
so $(K\times\A')_{S}$ is strictly unital since it is $\Ain$equivalent to $\A$ which is strictly unital.\ 
In particular, the unit of $(x,\F^0(x))\in\mbox{Ob}\big((K\times\A')_{S}\big)$ is given by $(1_{x},1_{\F^0(x)})$, where $1_x$ is the unit of $x\in\A$ and $1_{\F^0(x)}$ is the unit of $\F^0(x)\in\A'$.\
It is not difficult to prove that $(1_x,1_y)$ is the unit of the object 
$(x,y)\in\mbox{Ob}\big( (K\times\A'')_{S'} \big)$ by setting one of the morphisms $(k_i, a_i’')$ to be $(1_x, 1_y)$ in equation (\ref{condizione12}).\
In the same vein, $\mbox{Id}_{K}\times\G$, defined in (\ref{morphi}), and $\mbox{pr}^1_{\A''}$ are strictly unital $\Ain$functors.
\end{proof}

\begin{proof}[Proof of Theorem \ref{teoremino}]
It is clear that  $\mbox{pr}^1_{\A''}$ is a fibration (and full), it remains to prove that if $\F$ is an acyclic fibration, then $\mbox{pr}^1_{\A''}$ is an acyclic fibration.\
First we note that, if $\F$ is an acyclic fibration then even $\mbox{pr}^1_{\A'}:(K\times\A')_{S}\twoheadrightarrow \A'$ is so.
\begin{itemize}
\item[IsoFib)] For every $(k,a'')\in\mbox{Ob}\big( (K\times\A'')_{S'}\big)$, if there exists an isomorphism $\phi: a''\to \tilde{a}''$ in $\A''$, then there exists $\tilde{\phi}:(k,a'')\to (\tilde{k},\tilde{a}'')$, an isomorphism in $(K\times\A'')_{S'}$ such that $\mbox{pr}^1_{\A''}(\tilde{\phi})=\phi$.\\ 
We have that $\tilde{\phi}=(\tilde{\psi},\phi)$, we have to find $\psi:k\to\tilde{k}$.\ 
First we have $\G^1(\phi): \G^0(a'')\to\G^0(\tilde{a}''),$ isomorphism in $H(\A')$.\ 
Taking $(k,\G^0(a''))\in\mbox{Ob}\big((K\times\A')_{S}\big)$ then 
\begin{align*}
\xymatrix{
[\mbox{pr}^1_{\A'}]^0\big( (k,\G^0(a'')) \big) \ar[rr]^-{\sim}_-{G^1(\phi)} && \G^0(\tilde{a}'').
}
\end{align*}
Since $[\mbox{pr}^1_{\A'}]$ is isofibration then it exists $\Psi:=(\psi,\hat{\psi}):(k,\G^0(a''))\to (\tilde{k},\G^0(\tilde{a}''))$ such that $\hat{\psi}=\G^1(\phi)$.\
We take $\tilde{\phi}:=(\psi,{\phi})$ and we are done.
\item[FF)] 
If $\F$ is an acyclic fibration, then for every $x,y\in\A$ the complex $\mbox{Ker}(\F^1)(x,y)$ is acyclic.\ 
It follows that $[\mbox{pr}^1_{\A''}]$ is full and faithfull.
\item[ExSurj)] We take $z\in\mbox{Ob}(\A'')$, we have $\G^0(z)\in\mbox{Ob}(\A')$ then it exists $x\in \mbox{Ob}(\A)=\mbox{Ob}\big( (K\times\A)_{S} \big)$ such that $\F^0(x)\simeq \G^0(z)$ is isomorphic in $H(\A')$.\
Taking $z$, we have $(x,z)$ such that $\F^0(x)\simeq\G^0(z)$.\ 
Since $\F$ is an isofibration if $\alpha:[\F^0](x)\simeq w=\G^0(z)\in H(\A')$, then it exists $\phi$ such that $\F^1(\phi)=\alpha$.\ 
So $\G^0(z)=\F^0(\tilde{x})$ for some $\tilde{x}\in\mbox{Ob}(\A)$. 
\end{itemize}
\end{proof}

\end{document}